\documentclass[10pt]{article}
\usepackage[utf8]{inputenc}
\usepackage[english]{babel}
\usepackage[cal=rsfso]{mathalfa}
\usepackage[margin=1.2 in]{geometry}
\usepackage{hyperref}
\hypersetup{
    colorlinks=true,
    linkcolor=blue,
    citecolor=blue
}

\usepackage[backend=biber, maxbibnames=99,
giveninits=true, style=alphabetic]{biblatex}

\usepackage{mathtools}
\usepackage{amssymb}
\usepackage{graphicx}
\usepackage{bbm}
\usepackage{tikz-cd}
\usepackage{amsthm}
\usepackage{color}

\usepackage{enumerate}
\usepackage[shortlabels]{enumitem}

\usetikzlibrary{babel, arrows}

\newtheorem{theorem}{Theorem}
\newtheorem{proposition}[theorem]{Proposition}
\newtheorem{lemma}[theorem]{Lemma}

\theoremstyle{definition}

\newtheorem*{remark*}{Remark}

\usepackage{graphics}

\makeatletter
\newcommand*\wt[1]{\mathpalette\wthelper{#1}}
\newcommand*\wthelper[2]{%
        \hbox{\dimen@\accentfontxheight#1%
                \accentfontxheight#1 1.2\dimen@
                $\m@th#1\widetilde{#2}$%
                \accentfontxheight#1\dimen@
        }%
}

\newcommand*\accentfontxheight[1]{%
        \fontdimen5\ifx#1\displaystyle
                \textfont
        \else\ifx#1\textstyle
                \textfont
        \else\ifx#1\scriptstyle
                \scriptfont
        \else
                \scriptscriptfont
        \fi\fi\fi3
}
\makeatother

\newcommand{\EE}{\mathbb{E}}
\newcommand{\FF}{\mathbb{F}}
\newcommand{\RR}{\mathbb{R}}
\newcommand{\NN}{\mathbb{N}}
\newcommand{\ZZ}{\mathbb{Z}}
\newcommand{\PP}{\mathbb{P}}
\newcommand{\QQ}{\mathbb{Q}}

\newcommand{\moe}{\mathcal{O}}
\newcommand{\one}{\mathbbm{1}}

\newcommand{\fm}{\mathfrak{m}}

\newcommand{\MA}{\mathcal{A}}

\newcommand{\fqt}{\FF_q[\![t]\!]}
\newcommand{\lau}{\FF_q(\!(t)\!)}
\newcommand{\kdelta}{\mathcal{K}_\delta}

\usepackage{csquotes}

\title{The Kakeya Conjecture on Local Fields of Positive Characteristic}
\author{Alejo Salvatore}
\date{}

\newcommand{\Addresses}{{
  \bigskip
  \footnotesize
\noindent
\textsc{Department of Mathematics, University of Wisconsin, Madison
WI 53706}\par\nopagebreak
\noindent \textit{E-mail address}: \texttt{amsalvatore@wisc.edu}
}}

\begin{document}

\maketitle
\begin{abstract}
We state and prove an analogue of the Kakeya conjecture for the
local field $\lau$. This extends Arsovski's result on the Kakeya
conjecture to local fields of positive characteristic. We also prove the
Kakeya maximal conjecture in this setting.
\end{abstract}

\section{Introduction}

A \emph{Besicovitch set} or \emph{Kakeya set}
is defined as a compact set $E\subseteq \RR^n$
that contains a unit line segment in every direction.
The Kakeya conjecture states that every Besicovitch set in $\RR^n$
has Hausdorff dimension $n$. This problem has been thoroughly researched
due to its connections with harmonic analysis, differential equations,
arithmetic combinatorics and number theory, cf.
\cite{tao-rotating, bourgain-survey, wolff-survey}.
The conjecture remains open for $n\geq 3$.
Let $d(n)$ denote the least possible value for the Hausdorff
dimension of a Besicovitch set in $\RR^n$.
Currently the best lower bounds for $n=3,4$ are established in
\cite{katz-zahl-19} and \cite{katz-zahl-21}, where the authors show that
$d(3)\geq \frac{5}{2}+\varepsilon_0$ and $d(4)\geq 3.059$, respectively.
For $n=6$, the best known bound is $d(n)\geq 4$, proved by Wolff
in \cite{wolff-95}.
For the best known results in higher dimensions, see \cite{katz-tao},
\cite{hrz} and \cite{zahl-new}.
In regard to Minkowski dimension, 
Katz and Tao \cite{katz-tao} showed that every Besicovitch set in $\RR^n$
has Minkowski dimension at least $\frac{n}{\alpha}+\frac{\alpha-1}{\alpha}$
where $\alpha\approx 1.675$.

In view of the difficulty of this conjecture, 
Wolff proposed in \cite{wolff-survey} a finite field analogue of
the Kakeya conjecture as a toy model for the real case. He
asked whether it was possible to find constants $C_n>0$ such that
if $\FF_q$ is the finite field with $q$ elements,
then every Kakeya set $E\subseteq \FF_q^n$ has at least
$C_n q^n$ elements.
This was solved affirmatively by Dvir in his influential paper
\cite{dvir} using the polynomial method. His original proof gave
the value $C_n=\frac{1}{n!}$, which was subsequently improved to
$C_n=2^{-n}$ in \cite{dkss13} and to $C_n=2^{1-n}$ in
\cite{bukh-chao}.
However, it was noted in \cite{eot-2009} that the analogy between the
classical and finite field Kakeya conjectures is flawed in two ways.
Firstly, the bound obtained in
\cite{dvir} is too strong, as it says that all Kakeya sets in 
$\FF_q^n$ have, in a sense, positive measure.
In contrast, there are Besicovitch sets in $\RR^n$ which have
Lebesgue measure zero. Another deficiency of $\FF_q$ is that it
does not admit ``multiple scales'': that is, there is no natural notion of
distance on this ring that is nontrivial.
These problems are related to each other, since the existence of multiple
scales in $\RR^n$ is what allows for the construction of Besicovitch sets
with measure zero.
For this reason, in \cite{eot-2009}, Ellenberg, Oberlin and Tao
proposed analogues of the Kakeya conjecture for metric rings that have
multiple scales, specifically local
fields, their rings of integers and
quotient rings. The question of proving local field variants of the
Kakeya conjecture had already been investigated in the 1990s by J. Wright,
who gave a series of lectures in the University of New South Wales
on this topic.

In \cite{dummit-hablicsek}, Dummit and Hablicsek contructed
Kakeya sets of measure zero in $\lau^n$ for all $n\geq 2$, and they
also showed that Kakeya sets in $\fqt^2$ or $\ZZ_p^2$ have
Minkowski dimension $2$.
Fraser \cite{fraser} constructed Kakeya sets of measure zero in
$L^n$, where $L$ is any non-archimedean local field with finite
residue field.
Hickman and Wright \cite{hickman-wright} studied the Kakeya conjecture
on $\ZZ/N\ZZ$ in relation with the discrete restriction conjecture,
and they gave a simpler construction of a Kakeya set in
$\ZZ_p^n$ of measure zero.
Caruso \cite{caruso18} has shown that almost all Kakeya needle sets
have measure zero in the non-archimedean case.
These results show that non-archimedean local fields behave in a
similar way to $\RR$ in regard to Kakeya sets.

Recently Arsovski \cite{arsovski-21} proved the Kakeya conjecture 
for the $p$-adic field $\QQ_p$ using the polynomial
method. This result implies the Kakeya conjecture over any
finite field extension $L/\QQ_p$, so the Kakeya conjecture holds for any
non-archimedean local field of characteristic zero.
In this article we will adapt Arsovski's proof to the case
of local fields of positive characteristic. The main result of
this article is the following:

\begin{theorem}\label{teo-kakeya-fpt}
All $c$-Kakeya sets in $\lau^n$ have Hausdorff dimension $n$.
\end{theorem}
In order to do this one needs to find an appropriate
completely ramified extension of $\lau$ that is analogous to the
extension of $\QQ_p$ given by adjoining a primitive $p^k$-th root of unity.
This is achieved using well-known properties of Lubin--Tate extensions.

\subsection{Notation}

We write $A\gtrsim B$ if there is a constant $C>0$ (which may or
may not depend on $n$ or $q$) such that $A\geq CB$. We may also
add $n$, $q$ or another variable as a subscript in order to specify that
the implicit constant depends on those variables.
Throughout this article, none of the implicit constants will depend on $k$
(or on $\delta=q^{-k}$).

$H^s(\cdot)$ denotes the $s$-dimensional \emph{Hausdorff content}, which is
defined as
\[H^s(E)=\inf\Bigg\{\sum_{j=1}^\infty (\mathrm{diam}\,U_j)^s:\,
E\subseteq \bigcup_{j=1}^\infty U_j\Bigg\},\]
where $\mathrm{diam}\,U_j$ stands for the diameter of $U_j$. The
\emph{Hausdorff dimension} of a set $E$ in a metric space
is then defined as the smallest
real number $d=\mathrm{dim}_H(E)$ such that $H^s(E)=0$ for all
$s<d$, which also happens to be the greatest number such that
$H^s(E)=\infty$ for all $s>d$.

For any ring $R$, let $R[\![t]\!]$ be the ring of formal power
series
\begin{equation}\label{eq-a(t)}
a(t)=c_0+c_1t+\dots+c_nt^n+\dots,
\end{equation}
where all the coefficients $c_j$ are elements of $R$.
We shall denote by $K=\lau$ the field of formal Laurent series over
$\FF_q$, which is the quotient field of $A=\FF_q[\![t]\!]$.
Since every power series $a\in A$ with nonzero constant term has a
multiplicative inverse, it follows that every element of $K$ can be written
as $a(t)=t^{n_0}(c_0+c_1t+\dots)$ with $c_0\neq 0$ and $n_0\in\ZZ$.
Given a nonzero element $a\in K$, we define $v_t(a)$ as the smallest
integer $n_0$ such that the coefficient of $t^{n_0}$ in $a(t)$ is
nonzero. We also set $v_t(0)=-\infty$.
It is easy to show that the function $|a|=q^{-v_t(a)}$
is an absolute value that satisfies the \emph{ultrametric inequality}
\begin{equation}\label{eq-ultrametrica}
|a+b|\leq \max\{|a|,|b|\},
\end{equation}
for all $a,b\in K$. Therefore it induces a metric $d(x,y)=|x-y|$ on $K$.
We also equip $K^n$ with the supremum norm
\[|(a_1,\dots,a_n)|=\max_{1\leq j\leq n} |a_j|.\]
With this norm $K^n$ becomes a complete, locally
compact metric group with Hausdorff dimension $n$. Since
$K^n$ is locally compact, it has a Haar measure, which can be
normalized so that $|A|=1$.
We shall also use the notation $|\cdot|$ to denote the cardinality of
a finite set, however this is unlikely to generate confusion.

Given a field $F$ with an absolute value $|\cdot|$ satisfying
\eqref{eq-ultrametrica}, we define $\moe_F=\{x\in F:\,|x|\leq 1\}$, which
is called the \emph{valuation domain}, and
$\fm_F=\{x\in F:\,|x|<1\}$. It is easy to show that $\moe_F$ is
a subring of $F$ and that $\fm_F$ is an ideal of $\moe_F$.
Note that for every $x\in \moe_F$ with $|x|=1$, its inverse $x^{-1}\in F$
is also contained in $\moe_F$. Since every nonzero element of the
quotient ring $\moe_F/\fm_F$ is represented by an element $x\in \moe_F$
of absolute value $1$, it follows that $\moe_F/\fm_F$ is a field,
and it is called the \emph{residue field} of $F$.
In our case of interest $F=K$, we have $\moe_K=A$ and $\fm_K=tA$,
so the residue field is $A/tA\cong \FF_q$.
A \emph{non-archimedean local field} can be defined as a field $F$
with an absolute value $|\cdot|$ satisfying \eqref{eq-ultrametrica},
that is complete as a metric space and such that its residue field
$\moe_F/\fm_F$ is a finite field. Thus $\lau$ is a local field
for every prime power $q$.
The other main example of a local field is the field $\QQ_p$ of
$p$-adic numbers, where $p$ is a prime number. It is defined as
the quotient field of the ring
\[\ZZ_p=\ZZ[\![t]\!]/(t-p),\]
which is an integral domain. The elements of $\ZZ_p$ are called
\emph{$p$-adic integers}. It can be shown (cf. \cite[Chapter 1]{basic-nt} or
\cite[Thm. 9.16]{jacobson})
that every non-archimedean local field
is a finite extension of either $\FF_p[\![t]\!]$ or $\QQ_p$, for some
prime $p$. Conversely, for every finite separable extension $L$ of
either $\FF_p[\![t]\!]$ or $\QQ_p$, it is possible to extend the
absolute value of the base field uniquely so that $L$ becomes a local field
(cf. \cite[\S 9.8]{jacobson}).

Given $0<c\leq 1$, we define
a \emph{$c$-Kakeya set} in $K^n$ as a compact set $E\subseteq K^n$ such that
for any direction $w\in \fqt^n$ there is a set $J_w\subseteq \fqt$
of measure $\geq c$ and $b_w\in K^n$ such that $b_w+J_w.w\subseteq E$.

\medskip\noindent
\textbf{Acknowledgments.} I would like to thank Marcelo Paredes for a careful
reading of this article. I would also like to thank him,
Román Sasyk and the anonymous referee for their useful comments
and suggestions.

\section{Overview of the proof}

Both Arsovki's article and this one use Dvir's polynomial method,
which can be succintly described in the following way: given a
Kakeya set $E$ of small size, find a nonzero polynomial $g$ of low degree
that vanishes on $E$; then use the fact that $E$ contains
lines in many different directions to show that either $g$ or a related
polynomial is zero, arriving at a contradiction.

In order to find the polynomial, Dvir's original argument used
the following Lemma whose proof is based on dimension counting:

\begin{lemma}\label{lema-sin-nombre}
Let $k,n,d>0$ be integers.
Let $F$ be a field, and let $S\subseteq F^n$ be a set with less than
$\binom{n+d}{n}$ elements. Then there is a nonzero polynomial
$g\in F[X_1,\dots,X_n]$ of degree $\leq d$
such that $g(s)=0$ for all $s\in S$.
\end{lemma}

Afterwards, a method of multiplicities \cite{dkss13} was developed,
which requires a polynomial that vanishes with high multiplicity at
each point of $E$, and obtains better lower bounds in the
finite field case. However, the proof of Theorem \ref{teo-kakeya-fpt}
does not require a stronger result than Lemma \ref{lema-sin-nombre}.

It is possible to use the polynomial method for local fields since
the problem of bounding Hausdorff dimension can be discretized.
As we shall see in the next section,
the proof of Theorem \ref{teo-kakeya-fpt} essentially boils down down to
showing that every Kakeya set $E\subseteq (A/t^kA)^n$ has at least
$q^{kn}/u(k,n)$ elements, where $u(k,n)\lesssim_{\varepsilon}
q^{k\varepsilon}$ for every $\varepsilon>0$.
By treating $R=A/t^kA$ as an $\FF_q$-vector space of dimension $k$
and using lower bounds from the finite field case, one can show that
$|E|\geq 2^{-kn}q^{kn}$.
However, this is essentially the best bound obtainable by applying
the polynomial method over finite fields, since there are
examples (cf. \cite{saraf-sudan}) of Kakeya sets in $\FF_q^{kn}$ with 
$\gtrsim 2^{-kn}q^{kn}$ elements.
One problem with this approach is that polynomials with coefficients in
$\FF_q$ normally do not distinguish between an $R$-line and an
arbitrary affine subspace of dimension $k$ over $\FF_q$.
A subexponential bound for $u(k,n)$
was proved for $n=2$ in \cite{dummit-hablicsek}, but the argument
seemingly cannot be generalized to higher dimensions.

Arsovski's innovation consists of working with a primitive $p^k$-th root
of unity $\zeta$ in an algebraic closure of $\QQ_p$ and 
taking advantage of the group isomorphism $\zeta^{\ZZ}\cong \ZZ/p^k\ZZ$
to map the Kakeya set into a subset of $(\zeta^{\ZZ})^n
\subseteq \QQ_p(\zeta)^n$.
Arsovski then applies the polynomial method on $\QQ_p(\zeta)^n$.
This solves the aforementioned problem, and has the added benefit of
preserving some of the algebraic structure of $E$. This is because
every line in $R^n$ gets mapped into a translate of a subgroup
isomorphic to $\ZZ/p^k\ZZ$.

This strategy is not viable over fields of characteristic $p>0$, since
the only $p$-th root of unity in these fields is $1$. However,
the multiplicative group $\zeta^\ZZ$ can be reinterpreted as the
orbit of $\zeta$ under the action of the Galois group $\mathrm{Gal}(
\QQ_p(\zeta)/\QQ_p)$, and this idea also works in positive characteristic.
The Galois group needs to be isomorphic to $R=A/t^kA$. Fortunately,
the Existence Theorem in Local Class Field Theory tells us that
for every open subgroup $H\subseteq K^\times$ of finite index
there is a corresponding abelian extension $L/K$ with Galois group
$\mathrm{Gal}(L/K)\cong K^\times/H$.
Since the additive group $A/t^kA$ can be realized
as a quotient of $K^\times /H$ (for instance taking
$H=(1+t^kA)\, t^\ZZ$), this gives us the desired extension $L/K$.
It is easy to show that $L$ must necessarily be totally ramified.
For this type of abelian extensions,
Lubin--Tate theory gives a very concrete description of the Galois
action. More precisely, this action can be realized as the quotient
of an $A$-action on $L$ by power series. Moreover, there
is a polynomial $f_k$ whose splitting field is $L$, such that
its sets of roots $\Lambda_k$ is an additive group, and the elements
of $A$ acts as group endomorphisms, which makes $\Lambda_k$ an
$A$-module.
Addition allows for ``lines'' in $\Lambda_k^n$ passing through the origin
to be translated.

There is also a slight modification in the last part of the argument
which arrives at a contradiction, in relation to the finite field case.
In Dvir's argument, after homogenizing the polynomial so that it
can be evaluated in the
projective space $\PP^n(\FF_q)$, one uses degree considerations to conclude
that the homogenization vanishes at the hyperspace at infinity, which 
corresponds to the set of directions of affine lines in $\FF_q$. In
terms of the original polynomial, this means that its homogeneous
component of highest degree vanishes, so applying the Schwartz-Zippel
Lemma yields a contradiction. In Arsovski's method one works
over a finite extension $L$ of the local field. In this setting,
we shall reduce the coefficients of the polynomial modulo the
maximal ideal $\fm_L$ (which can be understood as ``evaluating at $t=0$'')
to obtain a polynomial $\overline{g}\in B[z_1,\dots,z_n]$, where
$B=\FF_q[X]$, that takes small values at many points $y\in C^n$,
where $C^n\subseteq B^n$ is a set that is in one-to-one correspondence
with the set of directions of $R^n$. Thus $C^n$ is an analogue
of the hyperspace at infinity, and the Schwartz-Zippel Lemma will be replaced
by Lemma \ref{lemma-sz}.

\section{Reduction to a covering theorem}

Let $A=\fqt$. 
A vector $w\in A^n$ is called \emph{primitive} or \emph{unitary}
if $|w|=1$,
this is equivalent to $w$ having some coordinate that is not a multiple of
$t$. The set of primitive vectors in $A^n$ will be denoted by
$S^{n-1}(A)$.  Likewise, if $R=A/t^kA$, a vector $w\in R^n$ is
\emph{primitive} if $w\notin (tR)^n$, and we define
$S^{n-1}(R)=R^n\smallsetminus (tR)^n$, the set of primitive vectors
in $R^n$. 

As in \cite{dvir} and \cite{arsovski-21},
we define an
$(\varepsilon,\nu)$-Kakeya set to be a set $E\subseteq K^n$
for which there is a set $\Omega\subseteq S^n(A)$ of directions
with measure $|\Omega|\geq \nu$,
such that for all $w\in \Omega$ there are $b_w\in K^n$ and
a subset $J_w\subseteq A$ of measure $\geq \varepsilon$
such that $b_w+wJ_w\subseteq E$. Similarly, we also say that
a subset $E\subseteq R^n$ is an $(\varepsilon,\nu)$-Kakeya set
modulo $t^k$ if there is a set $\Omega\subseteq S^{n-1}(R)$ with at least
$\nu q^{nk}$ elements, such that for each $w\in \Omega$ there is a line
$\ell_w\subseteq R^n$ in the direction of $w$ with
$|\ell_w\cap E|\geq \varepsilon q^k$.

Theorem \ref{teo-kakeya-fpt} will
be deduced from the following result:

\begin{theorem}\label{teo-aux}
Let $n,k>0$ be integers, and let $\varepsilon,\nu\in (0,1)$. Then
an $(\varepsilon,\nu)$-Kakeya set in $\fqt^n$ cannot be
covered by less than
\[\binom{\big\lfloor\tfrac{\nu\varepsilon}{kn}q^{k-1}\big\rfloor+n}{n}\]
closed balls of radius $q^{-k}$.
\end{theorem}

Theorem \ref{teo-aux} also implies the Kakeya maximal conjecture
for $\fqt$. As in the real case, this conjecture can be stated in
several equivalent ways. Here it is preferable to state it in terms of the
Kakeya maximal operator $\kdelta$, which is defined below.

Let $\delta=q^{-k}$, where $k>0$ is an integer.
A \emph{unit line segment} in $A^n$ is a set of the form $b+Aw$, where
$b,w\in A^n$ and $w$ is primitive. We define a
\emph{$\delta$-tube} to be the closed $\delta$-neighborhood of a unit line
segment, that is, a set of the form
\[T_\delta=T_\delta(b,w)=b+Aw+(t^kA)^n,\]
where $w$ is a primitive vector. As one might expect, these sets
have measure $\delta^{n-1}$. Given an integrable function
$\phi:A^n\to \RR$, we define its \emph{Kakeya maximal function}
$\kdelta\phi:S^{n-1}(A)\to \RR$ as
\begin{equation}
\kdelta\phi(w)=\sup_{b\in A^n} \frac{1}{\delta^{n-1}}
\int_{T_\delta(b,w)}|\phi|.
\end{equation}

\begin{theorem}[Kakeya maximal conjecture]\label{teo-max-fuerte}
Let $k>0$ be an integer, let $\delta=q^{-k}$ and let
$\phi:A^n\to \RR$ be a measurable function. Then
\begin{equation}\label{eq-teo-max-fuerte}
\|\mathcal{K}_\delta\phi\|_{L^n(S^{n-1}(A))}\lesssim_{n,q}
k^{n+2} \|\phi\|_{L^n(A^n)}^n.
\end{equation}
\end{theorem}

The constant depends logarithmically on $\delta=q^{-k}$,
which is also expected to happen in the real case. Similarly to
\cite{eot-2009}, the estimate \eqref{eq-teo-max-fuerte} will be deduced
from a distributional estimate
\begin{equation}\label{eq-maximal-dist}
|\{\omega\in S^{n-1}(R):\,\kdelta\phi(\omega)\geq \lambda\}|
\lesssim_{n,q}
k^{n+1} \lambda^{-n} \|\phi\|_{L^n(A^n)}^n,
\end{equation}
which will be shown to hold for all $\lambda>0$. The proof of 
\eqref{eq-maximal-dist} exploits the fact that there are only
finitely many tubes of a given size. Indeed, if
$w'=w+at^k$ for some $a\in A$, then $T_\delta(b,w')=T_\delta(b,w)$.
This means that $\kdelta\phi$ is constant
on each coset $b+(t^kA)^n$, and so it determines a well-defined function
on $S^{n-1}(R)$.
Moreover, each $\delta$-tube is a disjoint union of closed balls of
radius $\delta$. In fact, if $\pi:A^n\to R^n$ denotes the projection to the
quotient, the image of any $\delta$-tube
$T=T_\delta(b,w)$ is a line $\ell=\pi(b)+R\pi(w)$ and $T=\pi^{-1}(\ell)$.
Hence
\[\frac{1}{\delta^{n-1}}\int_T |\phi|=q^{-k}\sum_{v\in \ell}
\frac{1}{|B_v|}\int_{B_v}|\phi|\]
for any function $\phi$, where $B_v=\pi^{-1}(\{v\})$. Therefore
the values of $\kdelta\phi$ only depend on the averages
$|B_v|^{-1}\int_{B_v}|\phi|$. In view of these considerations,
the problem can be discretized as  follows:
given any function $\phi:R^n\to \RR$, let $\phi^*:S^{n-1}(R)\to \RR$
be the function
\[\phi^*(w)=\sup_{\ell\parallel w} \frac{1}{q^k}\sum_{v\in \ell}
|\phi(v)|,\]
where the supremum is taken over all lines $\ell\subseteq R^n$ that
point in the direction of $w$.
Then the bound \eqref{eq-maximal-dist} is equivalent to

\begin{theorem}[Kakeya distributional estimate]\label{teo-maximal-discreto}
Let $k>0$ be an integer, let $R=A/t^kA$ and let $\phi:R^n\to \RR$
be any function. Then
\begin{equation}\label{eq-maximal-discreta}
|\{w\in S^{n-1}(R):\,\phi^*(w)\geq \lambda\}|\lesssim_{n,q} k^{n+1}
\lambda^{-n}\|\phi\|_{\ell^n}^n,
\end{equation}
for all $\lambda>0$.
\end{theorem}

At the end of this section we will use Theorem \ref{teo-aux} to prove Theorem \ref{teo-maximal-discreto} followed by the proof of Theorem
\ref{teo-max-fuerte}.
One can deduce Theorem \ref{teo-kakeya-fpt} from the maximal conjecture
in an analogous way to the real case. However, it is worth mentioning that
Theorem \ref{teo-kakeya-fpt} can also be deduced directly
from Theorem \ref{teo-aux} using the following
well-known argument based on $q$-adic decomposition\footnote{
This argument can be found in \cite{wolff-survey}, for instance.}.
Let $\MA\subseteq K^n$ be a set of representatives
for the (countable) quotient group $(K/A)^n$. 
We may write any $c$-Kakeya set $E$ as a disjoint union
\[E=\bigcup_{a\in \MA} E(a),\]
where $E(a)=E\cap (a+A^n)$. Afterwards we can form the set
$E'=\bigcup_{a\in \MA} (E(a)-a)$. Clearly $\dim_H(E')\leq \dim_H(E)$
and it is easy to see that $E'$ is also a $c$-Kakeya set. Hence we
may assume that $E\subseteq A^n$. To prove that $\mathrm{dim}_H(E)=n$
we must show that $H^s(E)>0$ for all $s<n$.

Fix $s<n$. Let $\{U_j\}_{j=1}^\infty$ be
a covering of $E$ by bounded sets of diameter $\leq 1$. We cover
each $U_j$ by a ball $D_j=B(y_j,r_j)$ of radius $r_j=\mathrm{diam}\, U_j$.
Note that each $r_j$ is an integral power of $q$ or zero. For
$k\in\NN$, let 
\[\Sigma_k=\{j\in\NN:\, r_j=q^{1-k}\},\]
let $\nu_k=|\Sigma_k|$ and $E_k=E\cap\big(\bigcup_{j\in \Sigma_k} D_j\big)$.
For each $w\in A^n$, there exists $a_w\in E$ and a set
$J_w\subseteq A$ with $|J_w|\geq c$ such that $a_w+wJ_w\subseteq E$.
Let $\phi_w:J_w\to E$ be the map $\phi_w(x)=a_w+wx$.
By the pigeonhole principle, we may find an integer $k=k_w\geq 1$ such
that $|\phi^{-1}_w(E_{k})|\geq cc_1/k^2$, where $c_1=6/\pi^2$.
If we denote by $\Omega_m$ the set of points $w\in A^n$ such that
$k_w=m$, applying the pigeonhole principle again we can find $k\in \NN$
such that $|\Omega_k|\geq c_1/k^2$. This implies that $E_k$ is 
an $(\varepsilon,\delta)$-Kakeya set, where $\varepsilon=\delta=
\frac{c}{2k^2}$. By Theorem \ref{teo-aux} we have
\[\nu_k\geq \binom{\big\lfloor\tfrac{\delta\varepsilon}{kn}q^{k-2}
\big\rfloor+n}{n}\geq \frac{(\delta\varepsilon)^n q^{(k-2)n}}{(nk)^{n} n!}
\gtrsim_{n,q} \frac{q^{kn}}{k^{5n}}.\]
Then
\[\sum_{j=1}^\infty r_j^s \geq \nu_k q^{(1-k)s}\gtrsim_{n,q}
q^{k(n-s)}k^{-5n}\gtrsim_{n,q} 1.\]
The implicit constant does not depend on $k$, so this shows $H^s(E)>0$
for all $s<n$, and
therefore $\dim_H(E)=n$.

We now turn to the proof of Theorem \ref{teo-maximal-discreto}.
We first prove it for characteristic functions. Note the slightly
improved dependence on $k$.

\begin{proposition}\label{prop-restricted-weak}
For every subset $E\subseteq R^n$ the following inequality holds:
\begin{equation}\label{eq-maximal-caract}
|\{w\in S^{n-1}(R):\,\one_E^*(w)\geq \lambda\}|\lesssim k^n\lambda^{-n}
|E|,
\end{equation}
where $\one_E$ denotes the characteristic function of $E$.
\end{proposition}

\begin{proof}
Fix $\lambda>0$ and let
$\Omega=\{w\in S^{n-1}(R):\,\phi^*(w)\geq \lambda\}$.
First assume that $|\Omega|\geq \frac{1}{2}q^{kn}$.
For each $w\in \Omega$ there is a line $\ell_w$ in the
direction of $w$ such that $|E\cap \ell_w|=\phi^*(w)q^k\geq \lambda q^k$.
Hence $E$ is a $({\lambda},\alpha)$-Kakeya set modulo $t^k$,
where $\alpha=q^{-kn} |\Omega|\geq \frac{1}{2}$.
Applying theorem \ref{teo-aux} we get that
\[|E|\geq \binom{\big\lfloor\frac{\lambda}{4kn}q^{k-1}\big\rfloor+n}{n}
\gtrsim_n \frac{\lambda^n}{k^{n}} q^{(k-1)n}\gtrsim_{q,n}
\frac{\lambda^n}{k^{n}}|\Omega|.\]
For a general set, we use a random rotation trick.
Let $J=|\Omega|$ and let $m>0$ be the least integer such that
$m J\geq \frac{1}{2}q^{kn}$.
Pick $m$ random invertible matrices $R_1,\dots,R_m\in \mathrm{GL}_n(R)$
that are independent and set $\Omega'=\bigcup_{i=1}^m R_i(\Omega)$,
$E'=\bigcup_{i=1}^m R_i(E)$.
By independence, each $w\in S^{n-1}(R)$ is contained in $\Omega'$
with probability
\[1-\left(1-\frac{J}{|S^{n-1}(R)|}\right)^{\!m}\geq 
1-\left(1-\frac{2}{m}\right)^{\!m}\gtrsim 1.\]
By linearity of expectation, it follows that
$\EE|\Omega'|\gtrsim |S^{n-1}(R)|$, so we can choose specific values for
$R_1,\dots,R_m$ so that $|\Omega'|\gtrsim |S^{n-1}(R)|$. Fixing those
matrices, we also have that $E'$ contains a proportion of
at least $\frac{\lambda}{2}$ of the line $\ell_w$ for every
$w\in \Omega'$. By the previous part, we get
\[m|E|\geq |E'|\gtrsim \frac{\lambda^n}{k^{n}}
|\Omega'|\gtrsim \frac{\lambda^n m}{k^{n}}|\Omega|,\]
which becomes \eqref{eq-maximal-caract} after rearranging the terms.
\end{proof}

\begin{proof}[Proof of Theorem \ref{teo-maximal-discreto}]
We may assume that $\phi$ is non-negative.
Note that $\phi^*(w)\geq q^{-k(2-1/n)}\|\phi\|_{\ell^n}$ for
all $w\in S^{n-1}(R)$. Indeed, given such $w$ we can partition $R^n$
into $M=q^{k(n-1)}$ parallel lines $L_1,\dots,L_M$ in the direction of
$w$. Then
\begin{align*}
\phi^*(w)&= q^{-k}\max_{1\leq j\leq M}\|\phi\|_{\ell^1(L_j)}
\geq q^{-k}\Bigg(\frac{1}{M}\sum_{j=1}^M \|\phi\|_{\ell^1(L_j)}^n
\Bigg)^{\!\!1/n}\\
&\geq q^{-k}\Bigg(\frac{1}{M}\sum_{j=1}^M \|\phi\|_{\ell^n(L_j)}^n
\Bigg)^{\!\!1/n}\geq q^{-k(2-1/n)}\|\phi\|_{\ell^n(R^n)}.
\end{align*}
Thus we may assume that $\lambda\geq q^{-k(2-1/n)}\|\phi\|_{\ell^n}$
from now on. Consider the set
\[D=\left\{v\in R^n:\,\phi(v)\geq 2q^{-2k}\|\phi\|_{\ell^n}
\right\},\]
and let $\phi_D(v)=\phi(v)\one_D(v)$. Observe that
\[\sum_{v\in R^n\setminus D}\phi(v)^n<2^nq^{-2kn}\sum_{v\in R^n}
\|\phi\|_{\ell^n}^n=2^nq^{-kn}{\|\phi\|_{\ell^n}^n},\]
so $\|\phi-\phi_D\|_{\ell^n}<2q^{-k} \|\phi\|_{\ell^n}$ and therefore
$\|\phi_D\|_{\ell^n}>\frac{9}{10}\|\phi\|_{\ell^n}$. Similarly, it is
easy to show that
\[\phi^*(w)\leq \phi_D^*(w)+2q^{-2k}\|\phi\|_{\ell^n},\]
for every $w\in S^{n-1}(R)$, which implies that
\[\{w:\phi^*(w)\geq\lambda\}\subseteq \{w:\phi^*_D(w)\geq \tfrac{9}{10}
\lambda\}\]
since $\lambda\geq q^{-k(2-1/n)}\|\phi\|_{\ell^n}$. Replacing
$\phi$ with $\phi_D$ we may assume that the range of $\phi$ is
contained in $\{0\}\cup (q^{-2k}\|\phi\|_{\ell^n},\|\phi\|_{\ell^n}]$.
For $0\leq j<2k$ define
\[E_j=\{v\in R^n:\,q^{-j-1}\|\phi\|_{\ell^n}< \phi(v)\leq q^{-j}
\|\phi\|_{\ell^n}\},\]
and let $\psi_j=q^{-j}\|\phi\|_{\ell^n}\one_{E_j}$. Then the function
$\psi=\sum_{j= 0}^{2k-1} \psi_j$ satisfies
$\frac{1}{q}\psi(v)<\phi(v)\leq \psi(v)$ for all $v$.
In particular, $\frac{1}{q}\|\psi\|_{\ell^n}<
\|\phi\|_{\ell^n}\leq \|\psi\|_{\ell^n}$ and 
$\frac{1}{q}\psi^*(w)<\phi^*(w)\leq \psi^*(w)$ for all $w\in S^{n-1}(R)$.
For each $j$ let $c_j=\frac{1}{u}\|\psi_j\|_{\ell^n}^{n/(n+1)}$, where
\[u=\sum_{j=0}^{2k-1}\|\psi_j\|_{\ell^n}^{\frac{n}{n+1}}.\]
Since $\sum_{j=0}^{2k-1}c_j=1$, we have
\[\{w:\,\psi^*(w)\geq \lambda\}\subseteq \bigcup_{j=0}^{2k-1}
\{w:\,\psi_j^*(w)\geq c_j\lambda \},\]
so applying Proposition \ref{prop-restricted-weak} to each $\psi_j$ we get
\begin{align*}
|\{w:\,\psi^*(w)\geq \lambda\}|&\lesssim
\frac{k^n}{\lambda^n} \sum_{j=0}^{2k-1}\frac{\|\psi_j\|_{\ell^n}^n}{c_j^n}
=\frac{k^nu^n}{\lambda^n} \sum_{j=0}^{2k-1}\|\psi_j\|_{\ell^n}^{
\frac{n}{n+1}}\\
&=\frac{k^n}{\lambda^n} \Bigg(\sum_{j=0}^{2k-1}\|\psi_j\|_{\ell^n}^{
\frac{n}{n+1}}\Bigg)^{\!\!n+1}\leq \frac{k^n}{\lambda^n}
(2k)^{\frac{n}{n+1}}\sum_{j=0}^{2k-1}\|\psi_j\|_{\ell^n}^{n}
\lesssim \frac{k^{n+1}}{\lambda^n}\|\psi\|_{\ell^n}^n,
\end{align*}
where the penultimate step uses H\"older's inequality. Finally, the
estimate \eqref{eq-maximal-discreta} follows from the last inequality since
$|\{w:\phi^*(w)\geq \lambda\}|\leq |\{w:\psi^*(w)\geq \lambda\}|$ and
$\|\psi\|_{\ell^n}\leq q\|\phi\|_{\ell^n}$.
\end{proof}

\begin{proof}[Proof of theorem \ref{teo-max-fuerte}]
Use the identity
\[\|\mathcal{K}_\delta\phi(w)\|_{L^n(S^{n-1}(A))}=n\int_0^\infty |\{w:
\,\mathcal{K}_\delta\phi(w)\geq \lambda\}|\,\lambda^{n-1}d\lambda.\]
Let $C(n,q)>0$ be the implicit constant in \eqref{eq-maximal-dist}.
In the proof of Theorem \ref{teo-maximal-discreto} it was shown that
$q^{-k(2-1/n)}\|\psi\|_{\ell^n} \leq \psi^*(v)\leq \|\psi\|_{\ell^n}$
for every function $\psi:R^n\to \RR$. Since $\kdelta \phi$ can be
expressed in terms of a discrete maximal function associated to
$\phi$, it follows that the same holds for $\kdelta \phi$.
Then we have
\begin{align*}
\|\mathcal{K}_\delta\phi(w)\|_{L^n(S^{n-1}(A))} &=n\int_{0}^{\|
\phi\|} |\{w:\,\mathcal{K}_\delta\phi(w)\geq \lambda\}|
\,\lambda^{n-1}d\lambda\\
&\leq n\int_0^{q^{-2k}\|\phi\|} \lambda^{n-1}\,d\lambda+
n \int_{q^{-2k}\|\phi\|}^{\|\phi\|} C(n,q)k^{n+1}
\|\phi\|^n \,\frac{d\lambda}{\lambda}\\
&\leq \|\phi\|^n +nC(n,q)k^{n+1}\|\phi\|^n \ln(q^{2k})\\
&=(1+2n\ln(q)C(n,q)k^{n+2})\|\phi\|^n.
\qedhere
\end{align*}
\end{proof}

\section{A variant of the Schwartz-Zippel Lemma}

It remains to prove Theorem \ref{teo-aux}. For this we shall
need Lemma \ref{lemma-sz}, which is an adaptation of Lemma $6$ from
\cite{arsovski-21} to our case of interest.
In order to use the polynomial method we need to introduce an
additional variable $X$. This will be the main variable, while
all polynomials and power series in $t$ should be treated as scalars,
since they belong to the base ring $A$.
Let $B=\FF_q[X]$ and $A_k=\FF_q\oplus\dots \oplus \FF_q t^{k-1}$.
We associate to
each $a=a_0+a_1t+\dots+a_{k-1}t^{k-1}\in A_k$ a polynomial
$s_a=\sum_{j=0}^{k-1}a_jX^{q^j}\in B$. Let
$C=\{s_a:\, a\in A_k\}$.

\begin{lemma}[Discrete valuation Schwartz-Zippel lemma]\label{lemma-sz}
Let $f\in B[z_1,\dots,z_n]$ a nonzero polynomial whose leading term with
respect to the lexicographical ordering is $c_\alpha z^\alpha$, 
where $\alpha=(\alpha_1,\dots,\alpha_n)$ with $\alpha_j<q^{k}$ for all $j$.
Then for all $0<\theta\leq 1$, the number of elements $y\in C^n$ such that
$v_X(f(y))\geq v_X(c_\alpha)+\theta n q^k$ is less than
\[\max\{q^{nk},|\alpha| kq^{k(n-1)+1}/\theta\},\]
where $|\alpha|=\alpha_1+\dots+\alpha_n$.
\end{lemma}

\begin{proof}
We proceed by induction on $n$.
Starting with the base case $n=1$, assume for the sake of contradiction
that there is
a set $L\subseteq A_k$ of size $\geq \alpha qk/\theta$ such that
$v_X(f(y))\geq v_X(c_\alpha)+ \theta q^k$ for all $y\in C_L$.
Let $d=\big\lceil \!\log_q\left(\frac{k}{\theta}\right)\!\big\rceil\geq 1$.
Then $L$ must intersect at least $q^{-d}|L|>\alpha$ congruence classes
modulo $t^{k-d}$, so we may find $L'\subseteq L$ of size
$\alpha+1$ that does not contain a pair of elements that are
congruent modulo $t^{k-d}$. By Lagrange
Interpolation\footnote{This identity holds for polynomials over
any field, even in positive characteristic, cf.
\cite[\S 1.6]{friedberg-linear}} we have
\[\sum_{u\in L'}f(s_u)\prod_{w\in L'\setminus\{u\}}\frac{z_1-s_w}{s_u-s_w}
=f(z_1),\]
so looking at the degree $\alpha$ coefficients we obtain 
\[\sum_{u\in {L'}}f(s_u)\Bigg(\prod_{w\in {L'}\setminus\{u\}} (s_u-
s_w)\Bigg)^{\!\!-1}=c_\alpha.\]
Then there must exist $u_0\in {L'}$ such that
\[v_X\left(\textstyle \prod_{w\in {L'}\setminus\{u_0\}} (s_{u_0}-s_w)
\right)\geq \theta q^k.\]
However, note that $v_X(s_w-s_u)=v_X(s_{w-u})=q^{v_t(w-u)}$, so
\[v_X\Bigg(\prod_{w\in {L'}\setminus\{u_0\}} (s_{u_0}-s_w) \Bigg) =
\sum_{w\in {L'}\setminus\{u_0\}}q^{v_t(w-u_0)}.\]
For $0\leq j\leq k-d-1$, define
\[n_j=|\{w\in L'\setminus\{u_0\}:\, v_t(w-u_0)=j\}|,\]
\[N_j=|\{w\in L'\setminus\{u_0\}:\, v_t(w-u_0)\geq j\}|.\]
By Abel's summation formula we have
\begin{equation}\label{eq-2}
\sum_{w\in L'\setminus\{u_0\}} q^{v_t(w-u_0)}=\sum_{j=0}^{k-d-1}n_jq^j
=N_0+\sum_{j=1}^{k-d-1}N_j(q^j-q^{j-1}).
\end{equation}
Now note that $N_j\leq q^{k-d-j}$ since otherwise the pigeonhole
principle would give us two elements of $L'$ which are congruent modulo
$t^{k-d}$. And clearly $N_j\leq \alpha\leq q^r$, where
$r=\lceil \log_q(\alpha)\rceil$. For $j=0$ we have $N_0=\alpha\leq q^{k-d}$,
so $r\leq k-d$. Thus we may bound \eqref{eq-2} by
\begin{align*}
\sum_{w\in L'\setminus\{u_0\}} q^{v_t(w-u_0)}&\leq q^r+\sum_{j=1}^{k-d-r}
q^r(q^{j}-q^{j-1})+\sum_{j=k-d-r+1}^{k-d}q^{k-d-j}(q^{j}-q^{j-1})\\
&=q^{k-d}+\sum_{j=k-d-r+1}^{k-d}q^{k-d-j}(q^{j}-q^{j-1})\\
&< q^{k-d}+rq^{k-d}=(r+1)q^{k-d}\leq kq^{k-d}\leq \theta q^k.
\end{align*}
We have reached a contradiction, therefore $|L|<\alpha qk/\theta$. Now
assume that $n\geq 2$, and write $f$ as
\[f=z_1^{\alpha_1}g(z_2,\dots,z_n)+h(z_1,\dots,z_n),\]
where $\mathrm{deg}_{z_1}(h)<\alpha_1$. Fix a tuple $y'=
(y_2,\dots,y_n)\in C^{n-1}$
and consider two cases. If
\[v_X(g(y'))\geq v_X(c_\alpha)+\theta(n-1)q^k,\]
there are $\leq (\alpha_2+\dots+\alpha_n)kq^{k(n-2)+1}/\theta$ number of
$(n-1)$-tuples for which this can hold, which gives us at most
$(\alpha_2+\dots+\alpha_n)kq^{k(n-1)+1}/\theta$ tuples
$(y_1,\dots,y_n)\in C^n$. Suppose instead that
\[v_X(g(y'))< v_X(c_\alpha)+\theta(n-1)q^k.\]
Then $P(z)=f(z,y')$ is a degree $\alpha_1$ polynomial with
leading coefficient $g_2(y')$. The condition 
\begin{equation}\label{eq-1}
v_X(f(z,y'))\geq v_X(c_\alpha)+\theta(n-1)q^k
\end{equation}
has less solutions than $v_X(P(z))\geq v_X(g(y'))+\theta q^k$.
And this last inequality has at most $\alpha_1 kq/\theta$ solutions in
$z$ for each $y'$. Thus the number of tuples $(z,y')$ satisfying
\eqref{eq-1} is $\leq \alpha_1 kq^{k(n-1)+1}/\theta$ in the second case.
Adding the two bounds gives the result.
\end{proof}

\section{Lubin--Tate theory}

This section is dedicated to establish some results from Lubin--Tate theory
needed to prove Theorem \ref{teo-aux}, followed by the proof of said
Theorem. Lemmas \ref{lemma-lt-1} and \ref{lemma-lt-2} are special
cases of more general theorems in this area. The interested reader
may consult \cite[Chapter 1]{milneCFT}.

From now on let $f(X)=tX+X^q$. This is an additive polynomial,
as is satisfies the equation $f(X+Y)=f(X)+f(Y)$. The polynomial method
will be applied on the set of roots of an iterate of $f$. But
first we shall
construct the power series $[a]_f$ which will determine the
action of $R=A/t^kA$ on the set of roots.
Given two power series $a(X),b(X)$ with constant term $0$, one can
define its composition as follows: if $a=\sum_{j=1}^\infty c_jX^j$, then
\[a\circ b(X)=\sum_{j=1}^\infty c_jb^j(X).\]
Although this is an infinite sum, the computation of each coefficient
in the composition only requires a finite number of operations since
for any integer $m\geq 0$, the coefficient of $X^m$ in $b^j(X)$ becomes
$0$ once $j>m$. We shall denote by $b^{\circ m}$ the $m$-fold composition
of $b$ with itself.

\begin{lemma}\label{lemma-lt-1}
\begin{enumerate}[$(a)$]
	\item For every $a\in A$, there is a unique power series
$[a]_f\in A[\![X]\!]$ such that $[a]_f(X)\equiv aX\pmod{X^2}$ and
$[a]_f\circ f = f\circ [a]_f$. Moreover, $[a]_f$ is additive.
	\item $[a+b]_f=[a]_f+[b]_f$ and $[ab]_f=[a]_f\circ [b]_f$ for all
	$a,b\in A$. Moreover, $[c]_f=cX$ for all $c\in \FF_q$ and	
	$[t^m]_f=f^{\circ m}$ for every integer $m>0$.
	\item If $a=c_0+c_1t+\dots+c_{k-1}t^{k-1}\in A_k$ with $c_j\in \FF_q$,
then
\begin{equation}\label{eq-3}
[a]_f=\sum_{0\leq j<k} c_j f^{\circ j}(X)
\end{equation}
and $[a]_f\equiv s_a(X)\pmod{t}$.
\end{enumerate}
\end{lemma}

\begin{proof}
\begin{enumerate}[$(a)$]
	\item We define a sequence $(a_m)_{m\geq 0}$ recursively. Its first
term is $a_0=a$, and for $m>0$ we set
\begin{equation}\label{eq-recursive}
a_m=\frac{a_{m-1}^q-a_{m-1}}{t^{q^m}-t}.
\end{equation}
Note that all the elements of the sequence belong to $A$. Indeed,
since $t^{q^m-1}-1$ is a unit in $A$, it suffices to show that
$a_{m-1}^q-a_{m-1}$ is a multiple of $t$. If $c_0$ is the
degree-$0$ term of $a_{m-1}$, then the degree-$0$ term of
$a_{m-1}^q-a_{m-1}$ is $c_0^q-c_0=0$ since $c_0\in \FF_q$, which
proves the claim.

We now claim that $[a]_f=\sum_{m=0}^\infty a_mX^{q^m}$ has the desired
properties. The first congruence in $(a)$ holds since $a_0=a$.
In regards to the second one, we have
\begin{align*}
[a]_f\circ f&=\sum_{m=0}^\infty a_m(tX+X^{q})^{q^m}=
\sum_{m=0}^\infty a_m\Big(t^{q^m}X^{q^m}+X^{q^{m+1}}\Big)\\
&=a_0tX+\sum_{m=1}^\infty \big(a_mt^{q^m}+a_{m-1}\big)X^{q^m},
\end{align*}
whereas
\[f\circ [a]_f =\sum_{m=0}^\infty ta_mX^{q^m}+\sum_{m=0}^\infty
a_m^qX^{q^{m+1}}=a_0tX+\sum_{m=1}^\infty \left(ta_m+a_{m-1}^q\right)X^{q^m}.\]
These two power series are equal as a consequence of the recursive
equation \eqref{eq-recursive} defining $(a_n)_n$. Finally, $[a]_f$
is additive since $X^{q^m}$ is an additive polynomial for all $m\geq 0$.
	\item The power series $\phi(X)=[a]_f+[b]_f$ satisfies
$\phi(X)\equiv aX+bX\pmod{X^2}$ and
\[\phi\circ f=[a]_f\circ f+[b]_f\circ f=f\circ [a]_f+f\circ [b]_f
=f\circ \phi,\]
since $f$ is an additive polynomial. Hence, the
uniqueness of $[a+b]_f$ implies that $[a+b]_f=[a]_f+[b]_f$. The proof
of the equation $[ab]_f=[a]_f\circ [b]_f$ is similar.
If $c\in \FF_q$ then $f(cX)=tcX+c^qX^q=tcX+cX^q=cf(X)$, so $[c]_f=cX$.
Finally, note that for every integer $m>0$, the polynomial $f^{\circ m}(X)$
commutes with $f$ and its linear coefficient is $t^m$. It follows
that $[t^m]_f=f^{\circ m}$.
	\item Equation \eqref{eq-3} is a direct consequence of $(b)$. The
last congruence follows then from the fact that $f^{\circ j}(X)
\equiv X^{q^j}\pmod{t}$ for every integer $j\geq 0$.\qedhere
\end{enumerate}
\end{proof}

For general $a\in A$, define $P_a(X)=[\pi_k(a)]_f$, where
$\pi_k(a)\in A_k$ is the remainder
of $a$ modulo $t^k$. Let $K^{\textnormal{a}}$ be the algebraic
closure of $K$. Given an integer $k>0$, let $\Lambda_k\subseteq 
K^{\textnormal{a}}$ be the set of roots of $f^{\circ k}$ and let
$L=K(\Lambda_k)$ be the field generated by these roots. Note that
$\Lambda_k$ is an additive subgroup of $L$, since $f^{\circ k}$ is
an additive polynomial.
Recall that it is possible to extend the absolute value of $K$ to
$L$ in a unique way. Given a power
series $h=a_0+a_1X+\dots\in A[\![X]\!]$ and an element $x_0\in L$ with
absolute value $|x_0|<1$, it is possible to evaluate $h$ at $x_0$ since
the infinite sum
$h(x_0)=\sum_{j=0}^\infty a_jx_0^j$ converges.

\begin{proposition}\label{prop-eisenstein}
Let $h(X)=a_0+a_1X+\dots+a_mX^m\in A[X]$ be a polynomial such
that $|a_0|=q^{-1}$ and $|a_j|<1$ for all $1<j<m$, but
$|a_m|=1$. Then any root $\zeta\in K^{\textnormal{a}}$ of $h$ has
absolute value $|\zeta|=q^{-1/m}$.
\end{proposition}

\begin{proof}
See \cite[Prop. 7.55]{milneANT}.
\end{proof}

For every $k>0$, the polynomial $f^{\circ k}(X)$ satisfies the conditions
of Proposition \ref{prop-eisenstein}, so it is possible to evaluate
power series at points in $\Lambda_k$. The map $A\times \Lambda_k\to
\Lambda_k$, $a\cdot \zeta=[a]_f(\zeta)$ gives $\Lambda_k$ an
$A$-module structure. This follows from Lemma \ref{lemma-lt-1}.

\begin{lemma}\label{lemma-lt-2}
Let $K^{\textnormal{a}}$ be the algebraic closure of $K=\lau$,
let $\Lambda_k\subseteq K^{\textnormal{a}}$ be the set of
roots of $f^{\circ k}(X)$ and let $L=K(\Lambda_k)$. Then
\begin{enumerate}[$(a)$]
	\item $\Lambda_k$ is isomorphic to $A/t^kA$ as an $A$-module.
	\item $L/K$ is a totally ramified extension. That is,
if $\moe_L\subseteq L$ is the valuation domain of $L$ and $\fm_L$ is
its maximal ideal, then $\moe_L/\fm_L\cong \FF_q$.
\end{enumerate}
\end{lemma}

\begin{proof}
See \cite[Prop. 3.4, Thm. 3.6]{milneCFT}.
\end{proof}

\begin{proof}[Proof of Theorem \ref{teo-aux}]
Let $E\subseteq \fqt^n$ be an $(\varepsilon,\nu)$-Kakeya set.
Note that every closed ball of radius $q^{-k}$ in $A^n$ is a
coset of the form $a+t^kA^n$, $a\in A^n$. Hence the least number of closed
balls of radius $q^{-k}$ needed to cover $E$ is equal to the cardinality
of the projection $\wt{E}$ of $E$ onto $(A/t^kA)^n$. The Kakeya
property of $E$ implies that there is a subset $\Omega\subseteq 
(A/t^kA)^n$ of size at least $\nu q^{kn}$, such that for all
$w\in\Omega$ there exist $b_w=(b_{w,1},\dots,b_{w,n})\in (A/t^kA)^n$
and a set $J_w \subseteq 
A/t^kA$ of size at least $\varepsilon q^k$
such that $b_w+wJ_w\subseteq \wt{E}$.
Suppose for the sake of contradiction that 
$|\wt{E}|<\binom{\beta+n}{n}$, where $\beta=\lfloor\frac{\nu
\varepsilon}{kn}q^{k-1}\rfloor$.

Let $\zeta_1\in\Lambda_k$ be a generator of $\Lambda_k$ as an
$(A/t^kA)$-module, and consider the set
\[S=\{([s_1]_f(\zeta_1),\dots,[s_n]_f(\zeta_1)):\,(s_1,\dots,s_n)
\in \wt{E}\}.\]
By lemma \ref{lema-sin-nombre} there is a nonzero polynomial
$g(z_1,\dots, z_n)\in L[z_1, \dots, z_n]$ with degree at most $\beta$
that vanishes on $S$. We may assume that $g\in\moe_L[z_1,\dots,z_n]$
and that its reduction modulo $\fm_L$ is nonzero.
Given $w\in \Omega$, define $c_{w,j}=[b_{w,j}]_f
(\zeta_1)$ and
\[h_w(X)=g\left(c_{w,1}+P_{w_1}(X),\dots,c_{w,n}+P_{w_n}
(X)\right).\]
For all $a\in J_w$ we have
\begin{align*}
0 &=g\left([b_{w,1}+w_1a]_f(\zeta_1),\dots,[b_{w,n}+w_n
a]_f(\zeta_1)\right)\\
&= g\left(c_{w,1}+[w_1]_f(\zeta_a),\dots,c_{w,n}+[w_n
]_f(\zeta_a)\right)=h_w(\zeta_a),
\end{align*}
where $\zeta_a=[a]_f(\zeta_1)$. Thus
$\prod_{a\in J_w}(X-\zeta_a)\,|\,h_w(X)$, which implies
that the reduction $\overline{h}_w(X)\in \FF_q[X]$ modulo
$\fm_L$ satisfies $v_X(\overline{h}_w)\geq |J_w|\geq \varepsilon q^k$.
Since  $\zeta_1\in \fm_L$ and $[b_{w,j}]_f$ is a polynomial with constant
term equal to zero, it follows that also $c_{w,j}\in\fm_L$ for all $j$.
Therefore
\[\overline{h}_w(X)=\overline{g}\left(\overline{P}_{w_1}(X),\dots,
\overline{P}_{w_n}(X)\right)\!,\]
and each $\overline{P}_{w_j}\in C$. Applying Lemma \ref{lemma-sz}
to $\overline{g}$ with $\theta=\varepsilon/n$, we get
\[\nu q^{kn}\leq |\Omega|<\beta kq^{k(n-1)+1}\frac{n}{\varepsilon}
\leq \nu q^{kn},\]
which is absurd. This concludes the proof.
\end{proof}

\begin{remark*}
It is possible to use Lubin--Tate polynomials to prove the Kakeya
conjecture for $p$-adic fields, although some adjustments have to be made.
In first place, one can no longer take $f$ to be an additive polynomial,
since the only additive polynomials in characteristic zero are the linear
polynomials $f(X)=aX$ with constant term equal to zero. However, it
can be shown that if $f\in \QQ_p[X]$ satisfies $f(X)\equiv pX\pmod{X^2}$ 
and $f(X)\equiv X^p \pmod{p}$, then there is a unique formal group law
$F(X,Y)\in \QQ_p[\![X,Y]\!]$ such that $f(F(X,Y))=F(f(X),f(Y))$.
Instead of the additive property, the map $[\,\cdot\,]_f$ satisfies
$[a+b]_f=F([a]_f,[b]_f)$ for all $a,b\in \ZZ_p$.
Thus one should choose $f$ so that the
corresponding power series $F(X,Y)$ and $[a]_f$ take a relatively simple
form. A sensible choice is given by $f(X)=(1+X)^{p}-1$,
whose associated formal group law is $F(X,Y)=X+Y+XY=(1+X)(1+Y)-1$.
The roots of $f^{\circ k}=(1+X)^{p^k}-1$ are of the form
$1+\zeta^j$, where $0\leq j<p^k$ and $\zeta$ is a primitive
$p^k$-th root of unity. Therefore,
using this choice of $f$ one recovers Arsovski's proof.
\end{remark*}

\nocite{*}
\printbibliography[heading=bibintoc]

\Addresses
\end{document}